\documentclass[11pt]{amsart}
\usepackage{epsfig}
\usepackage{graphics}
\usepackage{mathrsfs}
\usepackage{epstopdf}
\usepackage{graphicx}
\usepackage{floatrow}
\usepackage[colorlinks]{hyperref}
\graphicspath{ {./images/} }
\usepackage{tikz}

\newtheorem{theorem}{Theorem}[section]

\newtheorem{lemma}[theorem]{Lemma}

\newtheorem{corollary}[theorem]{Corollary}

\theoremstyle{definition}

\newtheorem{remark}[theorem]{Remark}

\def\S{{\mathbb{S}}}

\theoremstyle{remark}

\def\mod{{\rm Mod}}

\begin{document}

\newenvironment{prooff}{\medskip \par \noindent {\it Proof}\ }{\hfill
$\square$ \medskip \par}
    \def\sqr#1#2{{\vcenter{\hrule height.#2pt
        \hbox{\vrule width.#2pt height#1pt \kern#1pt
            \vrule width.#2pt}\hrule height.#2pt}}}
    \def\square{\mathchoice\sqr67\sqr67\sqr{2.1}6\sqr{1.5}6}
\def\pf#1{\medskip \par \noindent {\it #1.}\ }
\def\endpf{\hfill $\square$ \medskip \par}
\def\demo#1{\medskip \par \noindent {\it #1.}\ }
\def\enddemo{\medskip \par}
\def\qed{~\hfill$\square$}

\date{\today}
 \title[Low-slope Lefschetz fibrations]
{Low-slope Lefschetz fibrations}

\author[Adalet \c Cengel, Mustafa Korkmaz]
{Adalet \c Cengel, Mustafa Korkmaz}


 \address{Department of Mathematics, Middle East Technical University,
 06800 Ankara, Turkey}
 \email{adalet.cengel@gmail.com}
\email{korkmaz@metu.edu.tr}

 \date{\today}
\thanks{Adalet Cengel was supported by TUBITAK postdoctoral fellowship BIDEB-2219, No:1059B191801519. }

\begin{abstract}
For $g\geq 3$, we construct genus-$g$ Lefschetz fibrations over the two-sphere whose slopes are arbitrarily close to $2$. The total spaces of the Lefschetz fibrations can be chosen to be minimal and simply connected. It is also shown that the infimum and the supremum of slopes all Lefschetz fibrations are not realized as slopes.
\end{abstract}

 \maketitle

 \setcounter{section}{0}

\section{Introduction}
By the work of Donaldson~\cite{Donaldson2},  every closed symplectic $4$-manifold, perhaps after blowing up, admits a Lefschetz fibration over the two-sphere. Conversely, Gompf~\cite{GompfStipsicz} showed that
the total space of genus-$g$ Lefschetz fibration is a symplectic $4$-manifold if $g\geq 2$. 
For a given product of positive Dehn twists representing the identity of the mapping class group of a closed oriented surface of genus $g$,  one can explicitly construct a genus-$g$ 
Lefschetz fibration over the two-sphere $\S^2$. 
Conversely,  every Lefschetz fibration with base $\S^2$ gives a positive factorization
 of the identity, unique up to Hurwitz moves and global conjugation. This gives a combinatorial way to study  
symplectic $4$-manifolds.

Let $f:X \rightarrow \S^2 $ be a (nontrivial) relatively minimal genus-$g$ Lefschetz fibration, where
$X$ is a closed oriented smooth $4$-manifold.  Let $ c_1^2 (X)$ and 
$\chi_h(X)$ denote the first Chern number and the holomorphic Euler characteristic of $X$, respectively. 
The \textit{slope}  of $f$ is defined as the quotient
\[
\lambda_{f}=\displaystyle \frac{c_1^2(X)+8(g-1)}{\chi_h(X)+(g-1)}.
\]

Xiao~\cite{Xiao} proved that relatively minimal holomorphic genus-$g$ Lefschetz fibrations over a
genus-$k$ curve satisfy the \textit{slope inequality} $4-4/g \leq\lambda_{f}$. 
Hain conjectured that for $g\geq 2$ every relatively minimal genus-$g$ Lefschetz fibration over $ \S^2$
satisfies the slope inequality (cf. {\cite[Conjecture 4.12] {en}} and \cite[Question 5.10]{abkp}).
In~\cite{MiyachiShiga} and~\cite{Monden}, several examples  of Lefschetz fibrations violating this
conjecture are constructed. In all of these examples, the slopes of the Lefschetz fibrations 
are close to $4-4/g$ for large $g$. 
We refer reader to~\cite{Monden} for a history of the slope conjecture.

The purpose of this paper is to construct genus-$g$ Lefschetz fibrations with base $\S^2$
having slope arbitrarily close to $2$ for every $g\geq 3$. All of the slopes we get are greater than $2$.
Moreover, the total spaces of these Lefschetz fibrations  can be chosen to be simply connected and minimal. 
We also show that the infimum and the supremum of the slopes of all Lefschetz fibrations cannot
be realized as the slope of any Lefschetz fibration.  See Section~\ref{sec:results}. 


Here is the idea of the proof of our result. Consider a  Lefschetz fibration $f_0$
of genus $g\geq 3$ with the base $\S^2$. Its monodromy contains $C^r$ for some Dehn twist $C$ 
about a nonseparating simple closed curve and for $r\geq 1 $, as was shown by Smith in~\cite{abkp, Smith}
that the monodromy cannot be contained in the Torelli group.
By taking a twisted fiber sum of the Lefschetz fibration with itself carefully,  one can guarantee that
the resulting monodromy contains a product $(DE)^r$, where $D$ and $E$ are positive Dehn twists about 
two disjoint curves cobounding a subsurface of genus $h$ in a regular fiber for any chosen 
$1\leq h\leq g-1$. Thus each
$DE$ can be replaced by a product of $(2h+1)(2h+2)$ Dehn twists about nonseparating curves
by using the odd chain relation. These two operations can be applied the resulting Lefschetz fibration repeatedly to get Lefschetz fibration $f_n$ for every $n$.
It turns out that the slopes of $f_n$ converge to $4h/(h+1)$, which is independent of the
initial choice $f_0$. See Theorem~\ref{thm:thm124}.

\medskip
\noindent
{\bf Acknowledgement:}
This paper is the improved version of the first author's PhD thesis, where 
a Lefschetz fibration of genus $g$  with slope  
\[
4-\displaystyle \frac{4}{g}- \displaystyle \frac{4(g-2)}{g(g+2)}
\]
was constructed for each $g\geq 3$.  We would like to thank \.Inan\c{c} Baykur, 
Noriyuki Hamada and Naoyuki Monden for their comments and 
their interests in this work.


\section{Preliminaries and Notations}
This section gives the necessary background and the known results used in our proofs. 

\subsection{Mapping Class Groups}
Let $\Sigma_{g,k}$ be a compact connected oriented surface of 
genus $g$ with $k\geq 0$ boundary components, and let ${\rm Mod}(\Sigma_{g,k})$ 
be the mapping class group of $\Sigma_{g,k}$, the group consisting of isotopy classes of 
orientation-preserving
diffeomorphisms of $\Sigma_{g,k}\to \Sigma_{g,k}$ fixing all  points on the boundary. 
Isotopies are also assumed to be the identity on the boundary. In this paper, we have
$k\leq 1$.
If $k=0$, we drop it from the notation and write $\Sigma_{g}$ and ${\rm Mod}(\Sigma_{g})$.

In this article, we always denote simple closed curves by the lowercase letters, and  
the positive (right) Dehn twist about them by the corresponding capital letters. If $f$ and $h$ are two diffeomorphisms
of $\Sigma_{g,k}$, the composition $fh$ means that $h$ is applied first. The conjugation
$fhf^{-1}$ is denoted by $h^f$.  All diffeomorphisms and curves are considered up to
isotopy. 

Throughout the paper, we fix the surface $\Sigma_{g,1}$ in Figure~\ref{fig:thesurface} and the curves
on it. The closed surface $\Sigma_g$ is the surface obtained from $\Sigma_{g,1}$ by gluing a 
disk along the boundary component. A curve on $\Sigma_{g,1}$ gives rise to a curve on $\Sigma_{g}$; 
these two curves are denoted by the same letter.

Dehn twists are basic mapping classes. A product of Dehn twists representing the identity
in ${\rm Mod}(\Sigma_{g,k})$ is called a \textit{relator}.
We first remind the following relations among Dehn twists. For the proofs the reader is referred to~\cite{ farbmargalit, Korkmaz4}.

\begin{itemize}
\item \textbf{Conjugation relation:} If $f\in \mod(\Sigma_{g,k})$ and if $c$ and $d$ are
 	two  simple closed curves on $\Sigma_{g,k}$ with $f(c)=d$, then $C^f=D$.
\item \textbf{Commutativity relation:} If $c$ and $d$ are two disjoint simple 
	closed curves on $\Sigma_{g,k}$, then
	\[
	CD=DC.
	\] 
\item \textbf{Braid relation:} If $c$ intersects $d$ transversely at one point, then
	\[
	CDC=DCD.
	\] 
\item \textbf{Lantern relation:} Let $a_1,a_2,a_3,a_4$ be four disjoint simple closed curves 
bounding a subsurface $\Sigma$  homeomorphic to a sphere with four boundary components.
Then there are three 
simple closed curves $x,y,z$ on $\Sigma$ such that the Dehn twists about them satisfy the 
lantern relation 
\[
  A_{1} A_{2} A_3 A_4=XYZ.
\]
The relator 
\begin{equation} \label{eqn:lanternrelator}
  \mathscr{L}=Z^{-1}Y^{-1}X^{-1}A_{1} A_{2} A_3 A_4
\end{equation} 
is a \textit{lantern relator}.

\item \textbf{Chain relation:} A \textit{chain} $(a_1,a_2, \ldots, a_n)$ of length $n$ is an ordered $n$-tuple 
	 of simple closed curves on  $\Sigma_{g,k}$ satisfying
\begin{itemize}
	\item[(i)] $a_i$ intersects  $a_{i+1}$ transversely at one point if $1\leq i \leq n-1$, and
	\item [(ii)]$a_i$ and $a_{j}$ are disjoint if $|i-j|>1.$
\end{itemize}

 If the length of the chain is odd, say $n=2h+1,$ a regular neighborhood of  
 $a_1 \cup a_2 \cup \cdots \cup a_{2h+1}$ is a genus-$h$ subsurface 
 of $\Sigma_{g,k}$ with two boundary components. If $d$ and 
 $e$ are two simple closed curves parallel to these 
  two boundary components, then, in ${\rm Mod}(\Sigma_{g,k})$,  we have the relation  
\[
(A_{1} A_2 \cdots A_{2h+1})^{2h+2}=DE,
\]
called an \textit{odd chain relation}~\cite{Wajnryb}. In this case,
\begin{equation} \label{eqn:oddrelator}
  \mathscr{C}_{2h+1}=(A_{1} A_{2} \cdots A_{2h+1})^{2h+2}D^{-1}E^{-1}
\end{equation} 
is an \textit{odd chain relator}.

If the length of the chain is even, say $n=2h,$ a regular neighborhood of 
$a_1 \cup a_2 \cup \cdots \cup a_{2h}$ is a genus-$h$ subsurface of $\Sigma_{g,k}$ 
with one boundary component. If $d$ is a simple closed curve parallel to 
this boundary component, then we have the relation 
\[
(A_{1} A_{2} \cdots A_{2h})^{4h+2}=D
\] 
in ${\rm Mod}(\Sigma_{g,k})$,  called an \textit{even chain relation}~\cite{Wajnryb}. We say that
\begin{equation} \label{eqn:evenrelator}
  \mathscr{C}_{2h}=(A_{1} A_{2} \cdots A_{2h})^{4h+2}D^{-1}
\end{equation} 
is an \textit{even chain relator}.

\item \textbf{Hyperelliptic relation:} 
Suppose that the surface is closed, i.e., $k=0$.
The Dehn twists about the curves of the maximal chain $(c_1, \ldots, c_{2g+1})$ on $\Sigma_g$ 
satisfy the relation 
\[
(C_{1} C_2 \cdots C_{2g}C_{2g+1}^2 C_{2g}\cdots C_2C_{1})^2=1,
\]
called the \textit{hyperelliptic relation}. The product
\begin{equation} \label{eqn:hyellrelator}
  h_g=(C_{1} C_2 \cdots C_{2g}C_{2g+1}^2 C_{2g}\cdots C_2C_{1})^2.
\end{equation} 
is called the \textit{hyperelliptic relator}.
\end{itemize}

\begin{figure}
\begin{tikzpicture}[scale=0.7]
\begin{scope} [yshift=0cm, scale=0.8]
 \draw[very thick, violet, rounded corners=10pt] (10.02,-1.9) -- (-3.7,-1.9)--
                                     	 (-4.6, -0.7)--(-4.6, 0.7)-- (-3.7,1.9) -- (10.02,1.9);

\draw[very thick, violet, rounded corners=6pt] (10,1.91) ..controls  (9.6,1.4) and (9.6,-1.4) ..(10,-1.91);
\draw[very thick, violet, rounded corners=6pt] (10,1.91) ..controls  (10.4,1.4) and (10.4,-1.4) ..(10,-1.91);
	 
\draw[very thick, violet] (-2.6,0) circle [radius=0.6cm];
\draw[very thick, violet] (0,0) circle [radius=0.6cm];
\draw[very thick, violet] (5.2,0) circle [radius=0.6cm];
\draw[very thick, violet] (7.8,0) circle [radius=0.6cm];
\draw[very thick, fill, violet] (2.2,0) circle [radius=0.03cm];
\draw[very thick, fill, violet] (2.6,0) circle [radius=0.03cm];
\draw[very thick, fill, violet] (3,0) circle [radius=0.03cm];

\draw[thick, red] (-2.6,0) circle [radius=0.9cm];
\draw[thick, red] (0,0) circle [radius=0.9cm];
\draw[thick, red] (5.2,0) circle [radius=0.9cm];
\draw[thick, red] (7.8,0) circle [radius=0.9cm];

\draw[thick, red, rounded corners=6pt, xshift=-3.9cm] (-0.7,-0.03) -- (-0.4,-0.2)--(0.4,-0.2) --(0.7,-0.03);
\draw[thick, red, dashed, rounded corners=6pt, xshift=-3.9cm] (-0.7,0.03) -- (-0.4,0.2)--(0.4,0.2) --(0.7,0.03);

\draw[thick, red, rounded corners=6pt, xshift=-1.3cm] (-0.7,-0.03) -- (-0.4,-0.2)--(0.4,-0.2) --(0.7,-0.03);
\draw[thick, red, dashed, rounded corners=6pt, xshift=-1.3cm] (-0.7,0.03) -- (-0.4,0.2)--(0.4,0.2) --(0.7,0.03);

\draw[thick, red, rounded corners=6pt, xshift=1.3cm] (-0.7,-0.03) -- (-0.4,-0.2)--(0.4,-0.2);
\draw[thick, red, dashed, rounded corners=6pt, xshift=1.3cm] (-0.7,0.03) -- (-0.4,0.2)--(0.4,0.2);

\draw[thick, red, rounded corners=6pt, xshift=3.9cm]  (-0.4,-0.2)--(0.4,-0.2) --(0.7,-0.03);
\draw[thick, red, dashed, rounded corners=6pt, xshift=3.9cm]  (-0.4,0.2)--(0.4,0.2) --(0.7,0.03);

\draw[thick, red, rounded corners=6pt, xshift=6.5cm] (-0.7,-0.03) -- (-0.4,-0.2)--(0.4,-0.2) --(0.7,-0.03);
\draw[thick, red, dashed, rounded corners=6pt, xshift=6.5cm] (-0.7,0.03) -- (-0.4,0.2)--(0.4,0.2) --(0.7,0.03);
\draw[thick, blue, rounded corners=4pt, xshift=-2.6cm] (0.02,0.6) -- (0.2, 0.9)--(0.2,1.6) --(0.02,1.9);
 \draw[thick, blue, dashed, rounded corners=4pt, xshift=-2.6cm] (-0.02,0.6) -- (-0.2, 0.9)--(-0.2,1.6) --(-0.02,1.9);
\draw[thick, blue, rounded corners=4pt, xshift=-2.6cm] (0.02,-0.6) -- (0.2, -0.9)--(0.2,-1.6) --(0.02,-1.9);
 \draw[thick, blue, dashed, rounded corners=4pt, xshift=-2.6cm] (-0.02,-0.6) -- (-0.2, -0.9)--(-0.2,-1.6) --(-0.02,-1.9);

\draw[thick, blue, rounded corners=4pt] (0.02,0.6) -- (0.2, 0.9)--(0.2,1.6) --(0.02,1.9);
 \draw[thick, blue, dashed, rounded corners=4pt] (-0.02,0.6) -- (-0.2, 0.9)--(-0.2,1.6) --(-0.02,1.9);
\draw[thick, blue, rounded corners=4pt] (0.02,-0.6) -- (0.2, -0.9)--(0.2,-1.6) --(0.02,-1.9);
 \draw[thick, blue, dashed, rounded corners=4pt] (-0.02,-0.6) -- (-0.2, -0.9)--(-0.2,-1.6) --(-0.02,-1.9);

\draw[thick, blue, rounded corners=4pt, xshift=5.2cm] (0.02,0.6) -- (0.2, 0.9)--(0.2,1.6) --(0.02,1.9);
 \draw[thick, blue, dashed, rounded corners=4pt, xshift=5.2cm] (-0.02,0.6) -- (-0.2, 0.9)--(-0.2,1.6) --(-0.02,1.9);
\draw[thick, blue, rounded corners=4pt, xshift=5.2cm] (0.02,-0.6) -- (0.2, -0.9)--(0.2,-1.6) --(0.02,-1.9);
 \draw[thick, blue, dashed, rounded corners=4pt, xshift=5.2cm] (-0.02,-0.6) -- (-0.2, -0.9)--(-0.2,-1.6) --(-0.02,-1.9);

\draw[thick, blue, rounded corners=4pt, xshift=7.8cm] (0.02,0.6) -- (0.2, 0.9)--(0.2,1.6) --(0.02,1.9);
 \draw[thick, blue, dashed, rounded corners=4pt, xshift=7.8cm] (-0.02,0.6) -- (-0.2, 0.9)--(-0.2,1.6) --(-0.02,1.9);
\draw[thick, blue, rounded corners=4pt, xshift=7.8cm] (0.02,-0.6) -- (0.2, -0.9)--(0.2,-1.6) --(0.02,-1.9);
 \draw[thick, blue, dashed, rounded corners=4pt, xshift=7.8cm] (-0.02,-0.6) -- (-0.2, -0.9)--(-0.2,-1.6) --(-0.02,-1.9);
\draw[thick, green, rounded corners=15pt] (-3.7,1.4)--(1.5,1.4)--(1.5, -1.4)--(-3.7,-1.4)--cycle;
\node[scale=0.8] at (1.5,1.4) {$u$};

\node[scale=0.8] at (-4.1 ,-0.5) {$c_1$};
\node[scale=0.8] at (-1.3 ,-0.6) {$c_3$};
\node[scale=0.8] at (-1.65 ,0.8) {$c_2$};
\node[scale=0.8] at (0.95 ,0.8) {$c_4$};
\node[scale=0.8] at (8.8 ,0.8) {$c_{2g}$};
\node[scale=0.8] at (-2.6 ,2.4) {$d_1$};
\node[scale=0.8] at (0 ,2.4) {$d_2$};
\node[scale=0.8] at (5.2 ,2.4) {$d_{g-1}$};
\node[scale=0.8] at (7.8 ,2.4) {$d_g$};
\node[scale=0.8] at (-2.6 ,-2.4) {$e_1$};
\node[scale=0.8] at (0 ,-2.4) {$e_2$};
\node[scale=0.8] at (5.2 ,-2.4) {$e_{g-1}$};
\node[scale=0.8] at (8.2 ,-2.4) {$e_g=c_{2g+1}$};
\node[scale=0.8] at (10, 2.4) {$\delta$};

\end{scope}
\end{tikzpicture}
\caption{The surface $\Sigma_{g,1}$ and the curves on it. The closed surface $\Sigma_g$ is 
obtained from $\Sigma_{g,1}$ by gluing  a $2$-disk along $\delta$.}\label{fig:thesurface}
\end{figure}
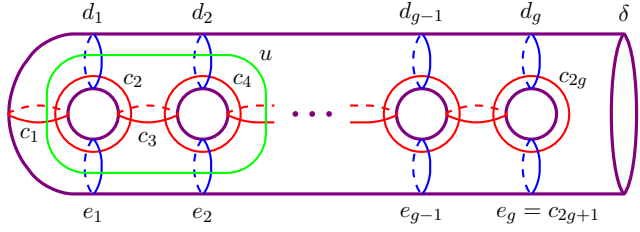


\subsection{Lefschetz fibrations} \label{sec:sec2-2}

Let $\S^2$ denote the $2$-sphere and 
let $X$ be a closed connected oriented smooth $4$-manifold.
A Lefschetz fibration on $X$ is a smooth surjective map $f:X\rightarrow \S^2$ 
with finite set of critical points $P=\{p_1,\ldots, p_n\}$  such that around 
each critical point $p_i$ and critical value  $f(p_i)$, there are orientation-preserving 
complex coordinate charts on which $f$ is of the form  $f(z_1,z_2)=z_1^2+z_2^2$. 
(In general, the base of a Lefschetz fibration can be a closed orientable surface, but
we only consider those with the base $\S^2$.)
We may assume that each singular fiber $f^{-1}(f(p_i))$ contains only 
 one singular point, which can be achived by a small perturbation of the fibration.
Throughout the paper, we assume that Lefschetz fibrations are \textit{nontrivial} and 
\textit{relatively minimal}, i.e. there is at least one singular fiber and no fiber contains a 
$(-1)$-sphere.

A regular fiber of $f$ is a closed connected orientable surface  $\Sigma_g$ of genus $g$.
The number $g$ is called the \textit{genus} of the Lefschetz fibration. 
A \textit{singular fiber} $f^{-1}(f(p_i))$ 
is obtained from a regular fiber by collapsing a simple closed curve $a_i$ on $\Sigma_g$, called a
\textit{vanishing cycle}, to a point. The diffeomorphism type of a regular
neighborhood of a singular fiber is 
determined by that vanishing cycle.

Let $f:X\rightarrow \S^2$ be a Lefschetz fibration.
For a fixed regular value $b_0,$ let us fix an identification of the 
regular fiber $f^{-1}(b_0)$ with the surface $\Sigma_g$. The 
Lefschetz fibration $f$ is then determined by a factorization
 \[
  A_1A_2\cdots A_n=1,
  \] 
called the \textit{monodromy} of the Lefschetz fibration, of the identity into positive Dehn twists in the mapping class group
  $\mod(\Sigma_g)$.
 The monodromy of $f$ is unique up to  a sequence of Hurwitz moves 
 and a global conjugation by a diffeomorphism. Recall that a Hurwitz move is 
\hspace{-0.5pt}
\[
A_1 A_2 \cdots A_iA_{i+1} \cdots A_n \sim A_1A_2 \cdots 
A_{i+1}^{A_{i}} A_{i} \cdots A_n, \]
or
\[
A_1A_2 \cdots A_{i} A_{i+1} \cdots A_n \sim A_1A_2 \cdots 
A_{i+1}  A_i^{A_{i+1}^{-1}} \cdots A_n.
\]

A \textit{section} of a Lefschetz fibration $f:X\rightarrow \S^2$ is 
a map $s:\S^2\rightarrow X$ such that the composition 
 $fs $ is the identity map of $\S^2$. 
 
It follows from from the standard theory of Lefschetz fibrations that if 
$f:X\rightarrow \S^2$ is a Lefschetz fibration with 
a regular fiber $\Sigma_g$ and with the 
monodromy $A_1 A_2 \cdots A_n =1$, the fundamental group $\pi_1(X)$ of $X$ is isomorphic
to a quotient of the group 
\begin{equation*}
\Gamma_f=\pi_1(\Sigma_g)/\langle a_1, a_2, \ldots, a_n \rangle;
\end{equation*}
$\pi_1(\Sigma_g)$ divided by the  normal closure of the vanishing cycles  (cf. \cite{GompfStipsicz}).
 Moreover, if $f$ has a section, then the groups $\pi_1(X)$ and $\Gamma_f$ are isomorphic.

For $i=1,2$, let $f_{i}: X_{i}\rightarrow \S^2$  be a Lefschetz fibration of 
genus $g$ with a regular fiber $\Sigma_g$ and with the monodromy 
$W_i$. After removing tubular neighborhoods $\nu\Sigma_g$
of $\Sigma_g$ from each $X_{i}$, we identify the boundaries of $X_{1}-\nu\Sigma_g$ 
and $X_{2}-\nu\Sigma_g$ via a fiber-preserving orientation-reversing diffeomorphism $\psi$. 
The (twisted) \textit{fiber sum} of $X_{1}$ and $X_{2}$ is  defined as
\[
f_1 \sharp_{\psi} f_2: X_{1} \sharp_{\psi} X_{2}\rightarrow \S^2,
\]
where $X_{1} \sharp_{\psi} X_{2}=(X_{1}-\nu\Sigma_g)\bigcup_{\psi}(X_{2}-\nu\Sigma_g).$
If $\psi$ is unimportant, then we drop it from the notation. The monodromy of the 
new Lefschetz fibration  is $W_1W_2^{\psi}$.

\begin{theorem}
[\cite{Usher,baykur1}] \label{thm:minimality}
Let $g\geq 2$. A fiber sum of any two  Lefschetz fibrations of genus $g$
is minimal.
\end{theorem}

A Lefschetz fibration $f$ is called \textit{holomorphic} if $X$ is a complex surface and  $f$ is a holomorphic map for a suitable complex structure on $\S^2$. 

The Euler characteristic of a Lefschetz fibration  $f: X\rightarrow \S^2$ can be computed as 
$e(X)=4-4g+n$, where $n$ is the number of singular fibers of $f$. Another invariant of $f$ is the  
 signature $\sigma(X)$ of $X$. There are various techniques to compute the signature: 
 see~\cite{Ozbagci,Smith,Endo,en,Matsumoto1}.
Endo and Nagami \cite{en} showed that the signature of a Lefschetz fibration 
can be calculated by using the signatures of relations contained in its monodromy. We use this method to calculate the signatures of  Lefschetz fibrations we construct.

Let $K_{f}^2=c_1^2(X)+8(g-1)$ and $\chi_{f}=\chi_h(X)+(g-1)$, where $c_1^2(X)$ is the first Chern
number and  $\chi_h(X)=(\sigma(X)+e(X))/4$ is the holomorphic Euler characteristic of $X$.
The \textit{slope} $\lambda_{f}$ of  $f$ is defined as the quotient
\[
\lambda_{f}=\displaystyle \frac{K_{f}^2}{\chi_{f}}.
\]
It follows from Lemma~$3.2$ in~\cite{Stipsicz99} that $K_{f}^2\geq 4g-4$.  Since the
$4$-manifold $X$ is symplectic~\cite{GompfStipsicz}, $b_2^+(X)\geq 1$.  The standard handle decomposition of $X$
has $n+2$ two-handles. By Corollary~$1.3$ in~\cite{Smith}, at least one of the vanishing cycles
is nonseparating. Thus, $1\leq b_2^+(X)+b_2^-(X)\leq n+1$. It follows now that
$\sigma(X)=b_2^+(X)-b_2^-(X) \geq 1-n$, i.e., $\sigma(X)+n \geq 1$. From the equality
\begin{eqnarray*}
\chi_h =\frac{\sigma(X)+e(X)}{4}=\frac{\sigma(X)+4-4g+n}{4}=\frac{\sigma(X)+n}{4}+1-g,
\end{eqnarray*}
we conclude that $\chi_f \geq 1$.

\subsection{Signatures of relations}
The signature of a Lefschetz fibration can be computed from the signatures of the relations involved 
in its monodromy.

\begin{theorem} 
[{\cite[Theorem 4.2]{en}}]
\label{c}
Let $f:X \rightarrow \S^2$ be a Lefschetz fibration of genus $g$ with the monodromy 
$A_1A_2 \cdots A_n=1.$ Then the signature of $X$ is 
\[
\sigma(X)=I_g(A_1 A_2 \cdots A_n).
\]
\end{theorem}

Here, $I_g$ is an integer-valued  function on the set of relators of the mapping class group of $\Sigma_g$, 
whose definition and properties can be found in~\cite{en}.

Let $a_1,a_2, \ldots, a_n$ and $b_1,b_2, \ldots, b_m$ be simple closed curves on $\Sigma_g$.
Suppose that the Dehn twists about them  satisfy the relation
\[
 A_1 A_2 \cdots A_n=B_1 B_2 \cdots B_m.
\]
For a positive relator $W=U \cdot A_1 A_2 \cdots A_n \cdot V$, a new positive relator
\[
W'=U \cdot B_1 B_2 \cdots B_m \cdot V
\]
is obtained by replacing the product $A_1 A_2 \cdots A_n$ in $W$ with $B_1 B_2 \cdots B_m$.  
If $R=B_1 B_2 \cdots B_m A_{n}^{-1} \cdots A_{2}^{-1} A_{1}^{-1}$, we say that $W'$ 
is obtained from $W$ by an \textit{R-substitution}. In this case, 
we also say that the Lefschetz fibration with monodromy $W'$ is obtained from 
the Lefschetz fibration with monodromy $W$.

\begin{theorem}
[{\cite[Theorem 4.3]{en}}]
\label{a}
Let $f:X \rightarrow \S^2$ and $f':X'\rightarrow \S^2$ be two genus-$g$ Lefschetz fibrations.
Suppose that  $f'$ is obtained from $f$ by an $R$-substitution. Then the
 signatures of $X$ and $X'$ satisfy the equality
\[
 \sigma(X')=\sigma(X)+I_g(R).
\]
\end{theorem}

The signatures of the relations we need are given below. For the proof, see~{\cite[Section 3] {en}}.

\begin{itemize}
\item \textbf{Lantern relator:}
 The signature of the lantern relator~\eqref{eqn:lanternrelator}
is
\[
I_g( \mathscr{L})=-1.
\]
\item \textbf{Odd chain relator}: The signature of the odd chain relator~\eqref{eqn:oddrelator}
is  
\[
I_g(\mathscr{C}_{2h+1})=-2h(h+2).
\]  

\item \textbf{Hyperelliptic relator}:  The signature of the hyperelliptic relator~\eqref{eqn:hyellrelator}
is 
\[
I_g(h_g)=-4(g+1).
\] 
\end{itemize}

We note that if $C$ and $D$ commute (resp. satisfy the braid relation), 
then the signature of the relator $CDC^{-1}D^{-1}$ (resp. $CDCD^{-1}C^{-1}D^{-1}$) 
  is $0$. Also, the signature of the even chain relator~\eqref{eqn:evenrelator} 
 is $I_g(\mathscr{C}_{2h})=-4h(h+1)+1$.


\section{Fiber sum,  substitutions and slope}

In this section, we determine the slope of the fiber sum of two Lefschetz fibrations. We also 
investigate how the slope changes under Lantern substitutions  
and $\mathscr{C}_{2h+1}$-substitutions. We remind that all Lefschetz fibrations we consider are nontrivial 
and relatively minimal, i.e. there is at least one singular fiber and no fiber contains a $(-1)$-sphere.

\begin{lemma} \label{lem:lemma111} 
Let $g\geq 2$ and 
let $f_1:X_1\rightarrow \S^2$ and $f_2:X_2\rightarrow \S^2$ be two 
 Lefschetz fibrations of genus $g$. For any fiber sum 
$f:X_1\sharp X_2 \rightarrow \S^2 $ of $f_1$ and $f_2$, we have
\begin{itemize}
\item[$(i)$]  $K_{f}^2 =K_{f_{1}}^2+K_{f_{2}}^2$ and 
\item[$(ii)$]  $\chi_{f} =\chi_{f_{1}}+\chi_{f_{2}}$.
\end{itemize}
In particular, $\lambda_{f}=(K_{f_{1}}^2+K_{f_{2}}^2)/(\chi_{f_{1}}+\chi_{f_{2}})$.
\end{lemma}
   
\begin{proof}
The Euler characteristic and the signature of the total space of $f$  are
	\begin{align*}
		e(X_1\sharp X_2)&=e(X_1)+e(X_2)+4(g-1) \mbox{ and}\\
		\sigma(X_1\sharp X_2)&=\sigma(X_1)+\sigma(X_2),
	\end{align*}
respectively. From these two equalities, we compute
	\begin{eqnarray*}
		K_{f}^2 &=&c_1^2(X_1 \sharp X_2)+8(g-1)\\
		&=& 3\sigma(X_1 \sharp X_2)+2 e(X_1 \sharp X_2)+8(g-1)\\
		&=& 3\sigma(X_1)+2e(X_1)+8(g-1)+
		    3\sigma(X_2)+2e(X_2)+8(g-1)\\
		&=& K_{f_{1}}^2+K_{f_{2}}^2,
	\end{eqnarray*}
and
	\begin{eqnarray*}
		\chi_{f} &=& \displaystyle \frac{\sigma(X_1 \sharp X_2)+e(X_1 \sharp X_2)}{4}+g-1\\
		&=& \displaystyle \frac{\sigma(X_1)+\sigma(X_2)+e(X_1)+e(X_2)+4(g-1)}{4}+g-1\\
		&=& \chi_{f_{1}}+\chi_{f_{2}}.
	\end{eqnarray*}
This proves the lemma.
\end{proof}

\begin{corollary} \label{cor:corol111}
Let $g\geq2$ and  let $f_1:X_1\rightarrow \S^2$ and $f_2:X_2\rightarrow \S^2$ be two 
Lefschetz fibrations of genus $g$ having the same slope $\lambda$. Then the slope of any 
fiber sum of $f_1$ and $f_2$ is also $\lambda$.
\end{corollary}

\begin{lemma} \label{lem:lemma9234} 
Let $g\geq 3$. Suppose that a genus-$g$ Lefschetz fibration $f_2:X_2\rightarrow \S^2$ is obtained from  
$f_1:X_1\rightarrow \S^2$ by a lantern substitution. Then $\lambda_{f_2}<\lambda_{f_1}$.
\end{lemma}
\begin{proof} Let $n_1$ and $n_2$ be the numbers of singular fibers of $f_1$ and $f_2$, respectively.
Since $f_2$ is obtained from $f_1$ by a lantern substitution 
\[
\mathscr{L}=Z^{-1}Y^{-1}X^{-1}A_1A_2A_3A_4,
\]
the signatures of $f_1$ and $f_2$ satisfies 
\[
\sigma(X_2)= \sigma(X_1)+I_g(\mathscr{L})=\sigma(X_1)-1,
\]
and $n_2=n_1+1$, so that
\[
e(X_2)  = 4-4g+n_2=4-4g+n_1+1=e(X_1)+1.
\]
Thus, 
\[
K_{f_2}^2 =K_{f_{1}}^2-1 \mbox{ and }  \chi_{f_2} =\chi_{f_{1}}.
\]

The lemma now follows.
\end{proof}

\begin{lemma} \label{lem:lemma131} 
Let $g\geq 3$. Suppose that a genus-$g$ Lefschetz fibration $f_2:X_2\rightarrow \S^2$ is obtained from  
$f_1:X_1\rightarrow \S^2$ by a $\mathscr{C}_{2h+1}$-substitution. Then 
\begin{itemize}
\item[$(i)$]  $K_{f_2}^2 =K_{f_{1}}^2+ 2h^2$ and 
\item[$(ii)$]  $\chi_{f_2} =\chi_{f_{1}}+\frac{1}2 (h^2+h)$.
\end{itemize}
\end{lemma}
\begin{proof}
Let $n_1$ and $n_2$ be the number of singular fibers of $f_1$ and $f_2$, respectively. Since $f_2$
is obtained from $f_1$ by a $\mathscr{C}_{2h+1}$-substitution, 
\[
n_2=n_1-2+(2h+1)(2h+2)= n_1+4h^2+6h.
\]
Therefore,
\[
e(X_2)  = 4-4g+n_2=e(X_1)+4h^2+6h
\]
and 
\[
\sigma(X_2)= \sigma(X_1)+I_g(\mathscr{C}_{2h+1})=\sigma(X_1)-2h^2-4h.
\]

The lemma follows from these two equalities.
\end{proof}

\begin{corollary} \label{cor:corol122} 
Let $g\geq 3$. Suppose that a genus-$g$ Lefschetz fibration $f:X\rightarrow \S^2$ is obtained from 
a fiber sum of 
the Lefschetz fibrations $f_1:X_1\rightarrow \S^2$ and $f_2:X_2\rightarrow \S^2$ followed by 
$\mathscr{C}_{2h+1}$-substitutions applied $r$ times. Then 
\begin{itemize}
\item[$(i)$]  $K_{f}^2 =K_{f_{1}}^2+K_{f_{2}}^2+ 2rh^2$, and 
\item[$(ii)$]  $\chi_{f} =\chi_{f_{1}}+\chi_{f_{2}}+\frac{1}2 r (h^2+h)$.
\end{itemize}
\end{corollary}


\section{Main results} \label{sec:results}
In this section we prove our main theorems.  We construct a sequence
of Lefschetz fibrations of genus $g\geq 3$ such that their slopes converge to $\frac{4h}{h+1}$
for any $1\leq h\leq g-1$. Taking $h=1$ gives our main theorem. The Lefschetz fibrations can 
be chosen to be simply connected and minimal. We then show that the infimum and the supremum of the slopes 
of all Lefschetz fibrations are not realized as the slope.

\begin{theorem} \label{thm:thm124} 
Let $g\geq 3$, $r\geq 1$ and $1\leq h\leq g-1$ be integers and let $f_0:X_0\rightarrow \S^2$ 
be a  genus-$g$ Lefschetz fibration
with the monodromy $VC^{r},$ where $C$ is a positive Dehn twist about a nonseparating 
curve and $V$ is a product of positive Dehn twists.
For each integer $n\geq 1$, there is a genus-$g$ Lefschetz fibration 
$f_n:X_n\rightarrow \S^2$ such that
\begin{itemize}
\item[$(i)$]  $K_{f_n}^2 =2^n K_{f_{0}}^2+ 2^n r\left[ (h+1)^n-1 \right] h$, and 
\item[$(ii)$]  $\chi_{f_n} =2^n\chi_{f_{0}}+2^{n-2} r\left[ (h+1)^n-1 \right] (h+1)$.
\end{itemize}
In particular, the slope of $f_n$ is
\begin{eqnarray*}
\lambda_{f_n}=\frac{4K^2_{f_0} + 4r\left[ (h+1)^n-1 \right] h}
                                 {4\chi_{f_{0}}+ r\left[ (h+1)^n-1 \right] (h+1)},
\end{eqnarray*}
so that its limit is  $\frac{4h}{h+1}$.
\end{theorem}
\begin{proof} Consider a regular fiber of $f_0$ and identify it with the surface $\Sigma_g$.
By a global conjugation of the monodromy, we may assume that $C=C_1$. Let us set $V_0=V$, $r_0=r$
and $W_0=V_0C_1^{r_0}$. Let $\phi_1$ and $\phi_2$ be two diffeomorphisms of the surface $\Sigma_g$ 
such that $\phi_1(c_1)=d_{h+1}$ and $\phi_2(c_1)=e_{h+1}$.  Note that the curves $d_{h+1}$ and $e_{h+1}$
are isotopic to the two boundary components of a regular neighborhood of $c_1\cup c_2\cup\cdots \cup c_{2h+1}$ 
on $\Sigma_g$, so that
\[
D_{h+1}E_{h+1}=(C_1C_2\cdots C_{2h+1})^{2(h+1)}.
\]

Suppose that for an integer $i\geq 0$, $f_i:X_i\rightarrow \S^2$ 
is a  Lefschetz fibration with a regular fiber $\Sigma_g$ 
and with the monodromy factorization $W_i=V_i C_1^{r_i}$,
where $V_i$ is a product of positive Dehn twists.  Then the product 
$W_i^{\phi_1} W_i^{\phi_2}$ may be written as 
\[
W_i^{\phi_1} W_i^{\phi_2}=V_i^{\phi_1}V_i^{D_{h+1}^{r_i}\phi_2} (D_{h+1}E_{h+1})^{r_i}.
\]

By trading each  $D_{h+1}E_{h+1}$ with $(C_1C_2\cdots C_{2h+1})^{2h+2}$,
we get a new factorization $W_{i+1}$ of the identity
 \begin{eqnarray} 
	W_{i+1} &=& V_i^{\phi_1}V_i^{D_{h+1}^{r_i}\phi_2}  (C_1C_2\cdots C_{2h+1})^{(2h+2)r_i} \label{eqn:eqnWi+1}  \\
   	   &=&	V_{i+1} C_1^{r_{i+1}},  \nonumber 
 \end{eqnarray}
where $r_{i+1}= 2(h+1)r_i$. Let $f_{i+1}:X_{i+1}\rightarrow \S^2$ be the Lefschetz fibration 
with the monodromy $W_{i+1}$. 

Notice that $f_{i+1}$ is obtained from $f_i$ by taking a twisted fiber sum with itself followed by 
$\mathscr{C}_{2h+1}$-substitutions performed $r_i$ times.
By Corollary~\ref{cor:corol122}, we have
\begin{eqnarray*}
 K_{f_{i+1}}^2 &=& 2K_{f_{i}}^2+ 2r_ih^2, \mbox{ and}\\
\chi_{f_{i+1}}  &=& 2\chi_{f_{i}}+\frac{1}2 r_i(h^2+h).
\end{eqnarray*}
It follows that  
\begin{eqnarray*}
 K_{f_{n}}^2 
	 	&=& 2^nK_{f_{0}}^2+ \left( \sum_{i=0}^{n-1} 2^{n-i} r_{i}\right) h^2\\
	 	&=& 2^nK_{f_{0}}^2+ \left( \sum_{i=0}^{n-1}  2^n(h+1)^i \, r_0 \right) h^2\\
		&=& 2^n K_{f_{0}}^2+ 2^n r\left[ (h+1)^n-1 \right] h.
\end{eqnarray*}
By a similar computation we get 
\[
\chi_{f_n} =2^n\chi_{f_{0}}+2^{n-2} r\left[ (h+1)^n-1 \right] (h+1).
\]

It is easy to check that the sequence $\lambda_{f_n}$  is
decreasing if $\lambda_ {f_0} > \frac{4h}{h+1}$  and  is increasing 
if $\lambda _{f_0} < \frac{4h}{h+1}$. In either case, 
its limit is $\frac{4h}{h+1}$. 

This completes the proof of the theorem.
\end{proof}

\begin{corollary} \label{cor:cor12} 
For each  $g\geq 3$, there is a Lefschetz fibration $f$ whose slope 
is greater than $2$ but arbitrarily close to $2$.  
\end{corollary}
\begin{proof}
Choose any Lefschetz fibration $f_0$ whose slope is greater than 
$2$, and apply Theorem~\ref{thm:thm124} for $h=1$.
\end{proof}


\begin{theorem} \label{thm:thm12} 
Let  $g\geq 3$. For every integer $n\geq 0$, there is a genus-$g$ 
Lefschetz fibration $F_n:Y_n\to \S^2$ such that
\begin{itemize}
\item[$(i)$]  the slope of $F_n$ satisfies $2< \lambda_{F_n} < 2+\frac{4g-8}{2^{n}}$, and
\item[$(ii)$]  the $4$-manifold $Y_n$ is minimal and simply connected.
\end{itemize}
\end{theorem}
\begin{proof}
In the mapping class group of the surface $\Sigma_{g,1}$ in Figure~\ref{fig:thesurface},
 it can easily be checked that the Dehn twist $\Delta$  about a curve parallel to boundary component
 $\delta$ may be written as
\begin{eqnarray} \label{eq:eq45} 
\Delta= (C_1 C_2 \cdots C_{2g}D_{g} E_{g} C_{2g} \cdots C_2C_1)^2.
\end{eqnarray} 
Gluing a disc along $\delta$  gives rise to a surjective homomorphism 
${\rm Mod}(\Sigma_{g,1})\rightarrow {\rm Mod}(\Sigma_g)$.  The relation 
 in (\ref{eq:eq45}) gives the hyperelliptic relation 
 \[
 W_0=(C_1 C_2 \cdots C_{2g}C_{2g+1}^2 C_{2g} \cdots C_2C_1)^2=1
 \]
  in ${\rm Mod}(\Sigma_g)$.  Recall that a curve on $\Sigma_{g,1}$ and 
  its image on $\Sigma_g$ are denoted by the same letter. 
 
 Let $f_0$ be the Lefschetz fibration with the monodromy $W_0$ and let $X_0$ be the total space of $f_0$. 
 The Euler characteristic of $X_0$ is $e(X_0)=4g+8$. Since
 $W_0$ is the hyperelliptic relator, the signature of $X_0$ is $\sigma(X_0)=-4g-4$, so that
 \[
K^2_{f_0}=4g-4
\]
and 
\[
\chi_{f_{0}}=g.
\]

Let $\phi_1$ and $\phi_2$ be two diffeomorphisms of the surface $\Sigma_{g}$ defined as
\[ 
\phi_1= U D_2C_{1}U \mbox{ and }\phi_2= U E_2C_{1}U,
\] 
so that $\phi_1(c_1)=d_2$ and $\phi_2(c_1)=e_2$.
By taking $r=h=1$, for each $n\geq 1$, Theorem~\ref{thm:thm124} gives 
a genus-$g$ Lefschetz fibration $f_n:X_n\rightarrow \S^2$ whose slope is
\[
\lambda_{f_n}=\frac{4K^2_{f_0} + 4\left( 2^n-1 \right) }
                                 {4\chi_{f_{0}}+ \left( 2^n-1 \right)2}
                       =\frac{2^{n+1}+8g-10}{2^n+2g-1}. 
 \]     
It is easy to check that $2< \lambda_{f_n} < 2+\frac{4g-8}{2^{n}}$.                          

Since the diffeomorphisms $\phi_1$ and $\phi_2$ of $\Sigma_g$ fix the simple closed curves
$c_3,c_6,c_7, c_8, \ldots, c_{2g+1}$,  for every $n\geq 1$, the monodromy $W_n$ contains the Dehn twists
$C_1,C_2,C_3,C_6,C_7,C_8, \ldots, C_{2g+1}$ (cf. the equality~\eqref{eqn:eqnWi+1}).

Now let $\psi$ be a self-diffeomorphism of $\Sigma_g$ such that $\psi(c_1)=c_4$ and 
$\psi(c_2)=c_5$. The word $W_nW_n^\psi$ is then a factorization of the identity
into positive Dehn twists containing $C_j$ for all $1\leq j\leq 2g+1$.
Let $F_n:Y_n\to \S^2$ be the Lefschetz fibration with the monodromy
$W_nW_n^\psi$, so that it is a fiber sum of $f_n$ with itself. 
By Corollary~\ref{cor:corol111}, the slope of $F_n$ is equal to $\lambda_{f_n}$.

Since $\lambda_{F_n}$ is equal to $\lambda_{f_n}$, it satisfies the inequalities in~$(i)$.
Since the curves $c_1,c_2,c_3,\ldots,c_{2g+1}$ are among vanishing cycles of $F_n$,
the $4$-manifold $Y_n$ is simply connected.  Moreover, the manifold $Y_n$ 
is minimal by Theorem~\ref{thm:minimality}.

This finishes the proof of the theorem.
\end{proof}        

\begin{remark}
It is natural to ask whether one may use other substitutions in the proofs of Theorems~\ref{thm:thm124} and~\ref{thm:thm12}.  The number of Dehn twist $C_1$
 appearing in the monodromy $W_n$ in above proofs grows exponentially with $n$. For this reason, 
 it seems that one cannot use the lantern relation instead of the odd chain relation. Even chain relation cannot be used  either
 because it involves a separating Dehn twist.
  We also note that in Theorem~\ref{thm:thm12}, each Lefschetz fibration $F_n$ has a section. 
 \end{remark}

\begin{theorem} \label{thm:izmir} 
Let  $g\geq 3$.  For every genus-$g$ Lefschetz fibration $f$, there is another 
Lefschetz fibration whose slope is less (resp. greater) than $\lambda_f$. In particular, there is no
Lefschetz fibration whose slope is  equal to the infimum (resp. supremum) of 
slopes of all Lefschetz fibrations.
\end{theorem}
\begin{proof} Let $\Sigma_g$ be a regular fiber of $f$.
 Since $g\geq 3$, there are nonseparating simple closed curves $a_1,a_2,a_3,a_4,x,y,z$ on 
$\Sigma_g$ such that the Dehn twist about them satisfy the lantern relation
$A_1A_2A_3A_4=XYZ$.

The monodromy $W$ of $f$ contains at least one Dehn twist $C$ about a nonseparating curve $c$.
Choose three diffeomorphisms $\varphi_1,\varphi_2,\varphi_3$ of $\Sigma_g$ such that
$\varphi_1(c)=x$, $\varphi_2(c)=y$ and $\varphi_3(c)=z$, so that $C^{\varphi_1}=X$, 
$C^{\varphi_2}=Y$ and $C^{\varphi_3}=Z$. By applying Hurwitz moves,  it is easy to see that
$W'=W^{\varphi_1}W^{\varphi_2}W^{\varphi_3}$ contains the factor $XYZ$, so that a lantern substitution
can be applied. The slope of the Lefschetz fibration $f'$ with monodromy $W'$ is equal to $\lambda_f$, as it is a fiber sum of three copies of $f$.  
By Lemma~\ref{lem:lemma9234}, the slope of the Lefschetz fibration obtained from $f'$ by a lantern substitution 
has slope less that $\lambda_f$.

In the same vein, by taking the fiber sum of four copies of $f$ if necessary, we may assume 
that the monodromy of contains the product $A_1A_2A_3A_4$ so that inverse of the lantern substitution may be applied to get a Lefschetz fibration with larger slope.

The second statement follows from the first.
\end{proof}

 \section{Final remarks}
 As mentioned in~\cite{Monden}, slopes of Lefschetz fibrations are bounded. For each $g\geq 2$, 
 let us define the functions $m_\lambda(g)$ and  $M_\lambda(g)$ to be, respectively, the infimum and the supremum of the set 
$\{ \lambda_f : f \mbox{ is a genus-$g$ Lefschetz fibration}\}$. From the discussions in Section~\ref{sec:sec2-2}, 
we have $m_\lambda(g)>0$. 

Ozbagci~\cite{Ozbagci} proved that $c_1^2(X)\leq 10 \chi_h+2g-2$
for every genus-$g$ Lefschetz fibration $f:X\to \S^2$. This inequality is equivalent to $\lambda_f\leq 10$, so that
that $M_\lambda (g)\leq 10$. But then we have $\lambda_f< 10$ by Theorem~\ref{thm:izmir}, so that 
$c_1^2(X)  < 10 \chi_h+2g-2$.

Since every genus-$2$ Lefschetz fibration is hyperelliptic,  
$m_\lambda(2)=2$. Moreover, this number is the slope of every genus-$2$ 
Lefschetz fibration without separating vanishing cycles (c.f.~\cite{Monden}).
Contrary to the $g=2$ case,  Theorem~\ref{thm:izmir} says that 
 for every Lefschetz fibration $f$ of genus $g\geq 3$, we have $m_\lambda (g) <\lambda_f<M_\lambda (g)$.
 
 It might be interesting to determine the numbers $m_\lambda(g)$ and $M_{\lambda}(g)$. Baykur and Hamada suspect that 
 $\lambda_f\geq 2$ for all genus-$g$ Lefschetz fibrations.
 In the light of this and of the results of this paper, 
 we conjecture that $m_\lambda(g)=2$ for all $g\geq 2$.

\end{document}